\documentclass[11pt]{article}

\usepackage[T1]{fontenc}
\usepackage[utf8]{inputenc}
\usepackage{lmodern}
\usepackage{amsmath,amssymb,amsthm,mathtools}
\usepackage{enumitem}
\usepackage[margin=1.15in]{geometry}
\usepackage{authblk}
\usepackage{booktabs}
\usepackage{tikz}
\usetikzlibrary{arrows.meta,calc,positioning}
\usepackage{microtype}
\usepackage[hidelinks,pdfpagelabels,hypertexnames=false]{hyperref}
\hypersetup{
  pdftitle={A complete classification of metrizable and strictly metrizable theta graphs},
  pdfauthor={Guangfu Wang},
  pdfsubject={Graph metrizability and theta graphs}
}
\allowdisplaybreaks

\newtheorem{theorem}{Theorem}[section]
\newtheorem{lemma}[theorem]{Lemma}
\newtheorem{proposition}[theorem]{Proposition}
\newtheorem{corollary}[theorem]{Corollary}

\theoremstyle{definition}
\newtheorem{definition}[theorem]{Definition}
\theoremstyle{remark}
\newtheorem{remark}[theorem]{Remark}

\newcommand{\cP}{\mathcal P}
\newcommand{\cQ}{\mathcal Q}

\begin{document}

\title{A complete classification of metrizable and strictly metrizable theta graphs}
\author{Guangfu Wang\thanks{Corresponding author: \texttt{gfwang@ytu.edu.cn}}}
\affil{School of Mathematics and Information Sciences, Yantai University,\\Yantai 264005, Shandong Province, China}
\date{}
\maketitle

\begin{abstract}
Cizma and Linial asked for a classification of the metrizable theta graphs.  We solve both their problem and its strict analogue.  For $a\le b\le c$, the theta graph $\Theta_{a,b,c}$ is metrizable if and only if $a\le 2$ or $(a,b,c)=(3,3,3)$, and it is strictly metrizable if and only if $a\le2$.  Thus $\Theta_{3,3,3}$ is precisely the exceptional theta graph that is metrizable but not strictly metrizable.  The negative directions follow from the known obstructions $\Theta_{3,3,4}$ and $\Theta_{3,3,3}$ together with topological-minor closure.  The positive direction is constructive.  Consistency turns the possible detours through a length-two arm into compatible Ferrers relations, which are represented by one-dimensional potentials; all resulting shortest-path comparisons have positive slack.  The exceptional ordinary-metrizable graph $\Theta_{3,3,3}$ is handled by a two-threshold weak Ferrers representation.  The proof is structural, yields rational edge lengths algorithmically, and uses no enumeration of path systems.
\end{abstract}

\noindent\textbf{Keywords:} metrizable graph; strictly metrizable graph; consistent path system; shortest path; theta graph; Ferrers relation.

\medskip
\noindent\textbf{2020 Mathematics Subject Classification:} 05C12, 05C22, 05C75, 05C85.

\section{Introduction}

Shortest paths in weighted graphs are usually studied from a metric viewpoint: one assigns positive lengths to the edges and then analyzes the resulting geodesics.  Cizma and Linial proposed the reverse problem \cite{CizmaLinial2022}.  Given a graph $G$, choose one simple path $P_{u,v}$ between every pair of vertices $u,v$.  The system is \emph{consistent} if every subpath of a chosen path is itself the chosen path between its endpoints.  It is \emph{metrizable} if there are positive edge lengths for which every prescribed path is a shortest path.  It is \emph{strictly metrizable} if those prescribed paths can all be made unique shortest paths.  The corresponding graph properties require every consistent path system to admit such a realization.

This definition isolates a graph-theoretic analogue of geodesic geometry.  It is close to, but distinct from, Bodwin's topological characterization of partial systems of unique shortest paths \cite{Bodwin2019}: graph metrizability concerns full path systems in a fixed graph and allows ties among shortest paths.  Cizma and Linial proved the structural results that make the notion tractable: cycles and outerplanar graphs are strictly metrizable, metrizability is closed downward under topological minors, the class has a finite forbidden-topological-minor characterization, and metrizability is decidable in polynomial time in principle \cite{CizmaLinial2022}.  Subsequent work constructed irreducible non-metrizable systems \cite{CizmaLinial2023}, developed a structure theorem for metrizable graphs \cite{ChudnovskyCizmaLinial2026}, showed that the strict variant is minor-closed \cite{ChudnovskyCizmaLinial2025}, and studied both the enumerative gap between consistent and metric path systems \cite{CizmaLinial2026Number} and quantitative approximation of non-metric systems \cite{CizmaLinial2026Approx}.

Open Problem 11.3 of \cite{CizmaLinial2022} asks for the classification of metrizable theta graphs.  The theta graph $\Theta_{a,b,c}$ consists of two branch vertices $s,t$ joined by three internally disjoint $s$--$t$ paths, called arms, of lengths $a,b,c$.  Since the parameters may be reordered, we assume $a\le b\le c$.  Cizma and Linial observed that $\Theta_{a,b,c}$ is metrizable when $a=1$, because such a graph is outerplanar, and they exhibited $\Theta_{3,3,4}$ as a non-metrizable obstruction.  Topological-minor closure then makes every $\Theta_{a,b,c}$ with $a\ge3$, $b\ge3$, and $c\ge4$ non-metrizable.  Thus the unresolved ordinary cases are exactly the infinite family with $a=2$ and the single graph $\Theta_{3,3,3}$.  In the strict setting, $\Theta_{3,3,3}$ is a known topological obstruction \cite[Fig.~13(9)]{ChudnovskyCizmaLinial2025}.

Our main theorem resolves these cases and also settles the strict analogue throughout the theta family.

\begin{theorem}\label{thm:main}
Let $a\le b\le c$.  Then
\[
        \Theta_{a,b,c}\text{ is metrizable}
        \quad\Longleftrightarrow\quad
        a\le2\text{ or }(a,b,c)=(3,3,3).
\]
Equivalently, $\Theta_{a,b,c}$ is non-metrizable if and only if
\[
        a\ge3,
        \qquad
        b\ge3,
        \qquad
        c\ge4.
\]
Moreover,
\[
        \Theta_{a,b,c}\text{ is strictly metrizable}
        \quad\Longleftrightarrow\quad
        a\le2.
\]
\end{theorem}

The theorem gives the promised complete answer for theta graphs.  Its significance is threefold.  First, it removes the last undetermined ordinary-metrizable cases left after the obstruction $\Theta_{3,3,4}$ and topological-minor closure: the infinite family with one arm of length $2$ and the exceptional graph $\Theta_{3,3,3}$.  Second, it completes the strict classification and shows that $\Theta_{3,3,3}$ is precisely the theta graph separating the two notions.  Third, the proof gives a structural realization mechanism rather than a computer-aided enumeration of consistent systems.  The mechanism is a Ferrers-potential construction: consistency turns the possible detours through the short arm into nested Ferrers relations, and these relations are then encoded by one-dimensional potentials whose successive differences become the desired edge lengths.

First we reduce to neighborly path systems, since a non-neighborly system on a theta graph avoids some edge and is therefore supported on an outerplanar graph.  If the length-two arm of $\Theta_{2,b,c}$ is not the prescribed $s$--$t$ path, a cycle-cut realization lemma reduces the problem to the metrizability of cycles.  The main case is therefore $P_{s,t}=s-a-t$.  There, consistency partitions each long arm into a prefix, a middle interval, and a suffix.  Same-arm prescribed paths through the short arm form difference Ferrers relations, while cross-arm choices in the middle rectangle form an additive Ferrers relation.  We realize these relations by potentials along the two long arms and then define the edge lengths from successive potential differences.  The graph $\Theta_{3,3,3}$ is exceptional but positive: the selected length-three arm supplies two nested cuts, and because each remaining middle interval has size at most two, a two-threshold Ferrers representation suffices.

A separate companion manuscript uses the strict theta theorem proved here as
an external input in a global forbidden-minor characterization of strictly
metrizable graphs \cite{WangStrictForbidden2026}.  That paper cites rather
than reproves the theta classification; conversely, the present article makes
no claim about the global forbidden-minor list.

The paper is organized as follows.  Section~2 contains definitions, notation, and the negative directions.  Section~3 proves a cut-respecting cycle realization theorem and the required Ferrers representation lemmas.  Section~4 proves strict metrizability of $\Theta_{2,b,c}$, ordinary metrizability of $\Theta_{3,3,3}$, and the main theorem.  Section~5 records the minimal theta obstructions and the algorithmic content of the construction.

\section{Preliminaries and reductions for theta graphs}

All graphs in this paper are finite and connected, and all paths are simple unless explicitly stated otherwise.  The length of a path is its number of edges.  If $Q$ is a path and $x,y$ are vertices of $Q$, then $Q[x,y]$ denotes the subpath of $Q$ with endpoints $x$ and $y$, read in either direction as needed.  Since our path systems are indexed by unordered vertex pairs, two opposite orientations of the same path are identified.

If $w:E(G)\to(0,\infty)$ is an edge weighting and $Q$ is a path, write
\[
        w(Q)=\sum_{e\in E(Q)}w(e).
\]
The corresponding distance is
\[
        d_w(x,y)=\min\{w(Q): Q\text{ is an }x\text{--}y\text{ path in }G\}.
\]
Thus a path is $w$-shortest, or a geodesic, if its $w$-length equals this minimum.  A weighting \emph{strictly} induces a path system if each prescribed path is strictly shorter than every other simple path with the same endpoints; ordinary metrizability in this paper allows ties.

We use the following notation for paths on theta graphs.  If
\[
        D: s=d_0,d_1,\ldots,d_m=t
\]
is an oriented $s$--$t$ arm and $x,y\in V(D)$, then $xDy$ denotes the subpath of $D$ from $x$ to $y$.  Concatenations such as $d_iDs\,a\,t$ mean the path obtained by first following the subpath $d_iDs$ and then the edges $sa$ and $at$; repeated vertices at concatenation points are suppressed.  A \emph{prefix} of $D$ means a set of the form $\{d_0,d_1,\ldots,d_p\}$, and a \emph{suffix} means a set of the form $\{d_q,d_{q+1},\ldots,d_m\}$.

\begin{definition}
Let $G=(V,E)$ be a graph.  A \emph{path system} $\cP$ in $G$ is a collection
\[
        \cP=\{P_{u,v}:\{u,v\}\in \binom{V}{2}\},
\]
where $P_{u,v}$ is a simple $u$--$v$ path.  The system is \emph{consistent} if, whenever $x$ and $y$ lie on $P_{u,v}$, the subpath $P_{u,v}[x,y]$ is the prescribed path $P_{x,y}$.  It is \emph{neighborly} if $P_{u,v}=uv$ for every edge $uv\in E$.
\end{definition}

\begin{definition}
A positive edge weighting $w:E\to(0,\infty)$ \emph{induces} a path system $\cP$ if every $P_{u,v}$ is a $w$-shortest $u$--$v$ path, and it \emph{strictly induces} $\cP$ if every $P_{u,v}$ is the unique $w$-shortest $u$--$v$ path.  A path system is \emph{metrizable}, respectively \emph{strictly metrizable}, if it is induced, respectively strictly induced, by some positive edge weighting.  A graph has either property if every consistent path system in it has the corresponding property.
\end{definition}

For positive integers $a,b,c$, the theta graph $\Theta_{a,b,c}$ consists of two branch vertices $s,t$ and three internally disjoint $s$--$t$ paths, called arms, of lengths $a,b,c$.  We always reorder the parameters as $a\le b\le c$.  In the case $\Theta_{2,b,c}$ the length-two arm is written
\[
        A=s-a-t,
\]
and the two remaining arms are usually written $B$ and $C$.

\begin{figure}[t]
\centering
\begin{tikzpicture}[x=1cm,y=1cm,baseline=(current bounding box.center)]
  \node[circle,fill,inner sep=1.7pt,label=left:$s$] (s) at (0,0) {};
  \node[circle,fill,inner sep=1.7pt,label=right:$t$] (t) at (7,0) {};
  \draw[thick] (s) .. controls (2,2.1) and (5,2.1) .. node[above] {$A$ (length $a$)} (t);
  \draw[thick] (s) -- node[above] {$B$ (length $b$)} (t);
  \draw[thick] (s) .. controls (2,-2.1) and (5,-2.1) .. node[below] {$C$ (length $c$)} (t);
\end{tikzpicture}
\caption{The theta graph $\Theta_{a,b,c}$ and the arm notation used throughout.}
\label{fig:theta-notation}
\end{figure}

We use the following facts from Cizma and Linial \cite{CizmaLinial2022} and Chudnovsky, Cizma, and Linial \cite{ChudnovskyCizmaLinial2025}.

\begin{proposition}\label{prop:CLfacts}
The following hold.
\begin{enumerate}[label=\textup{(\roman*)}]
\item Cycles are strictly metrizable.
\item Outerplanar graphs are strictly metrizable.
\item Metrizability is closed downward under topological minors.  Equivalently for our purposes, if a graph contains a subdivision of a non-metrizable graph, then it is non-metrizable.
\item Strict metrizability is closed downward under topological minors, and $\Theta_{3,3,3}$ is not strictly metrizable \cite[Fig.~13(9)]{ChudnovskyCizmaLinial2025}.
\end{enumerate}
\end{proposition}

We also use the cycle structure theorem of Cizma and Linial: after all persistent edges are contracted, every non-trivial consistent path system on a cycle becomes the standard geodesic system on an odd cycle \cite[Proposition 4.8]{CizmaLinial2022}.  Here the \emph{standard geodesic system} on an odd cycle means the system that chooses, for every pair of vertices, the shorter of the two cyclic arcs; on an odd cycle this shorter arc is unique.  The trivial cycle systems are those supported on a single spanning path of the cycle; they will not arise in the application of Lemma~\ref{lem:cycle_cut}, because there the restricted cycle system is neighborly.

The next reductions isolate the only theta path systems that require new work.

\begin{proposition}\label{prop:nonneighborly}
Let $G=\Theta_{a,b,c}$.  Every non-neighborly consistent path system on $G$ is strictly metrizable.
\end{proposition}

\begin{proof}
Let $\cP$ be non-neighborly.  Choose an edge $e=xy$ for which $P_{x,y}\ne xy$.  No path of $\cP$ can use $e$: if $e$ appeared in some $P_{u,v}$, then the $x$--$y$ subpath of $P_{u,v}$ would be the edge $xy$, and consistency would force $P_{x,y}=xy$.

Thus all paths of $\cP$ lie in $G-e$.  The graph $G-e$ is outerplanar, so by Proposition~\ref{prop:CLfacts} it has a positive edge weighting $w_0$ strictly inducing $\cP$ as a path system on $G-e$.  Extend $w_0$ to $G$ by assigning $e$ a weight larger than the total $w_0$-weight of all edges in $G-e$.  Every alternative using $e$ is then strictly longer than the prescribed path, while strict comparisons inside $G-e$ are unchanged.  Hence the extended weighting strictly induces $\cP$ on $G$.
\end{proof}

\begin{corollary}\label{cor:neighborly_only}
Any non-metrizable path system on a theta graph is neighborly.
\end{corollary}

\begin{proposition}\label{prop:nonmetro_side}
Let $a\le b\le c$.  If $a\ge3$, $b\ge3$, and $c\ge4$, then $\Theta_{a,b,c}$ is non-metrizable.
\end{proposition}

\begin{proof}
The graph $\Theta_{a,b,c}$ contains a subdivision of $\Theta_{3,3,4}$.  The graph $\Theta_{3,3,4}$ is non-metrizable by the certificate displayed as Fig. 22k in \cite{CizmaLinial2022}.  The conclusion follows from topological-minor closure, Proposition~\ref{prop:CLfacts}(iii).
\end{proof}

For reference, we record the local algebraic form of the obstruction.

\begin{proposition}\label{prop:sixblock}
Suppose a graph contains internally disjoint subpaths
\[
 A:s-a_1-a_2-t,\qquad
 B:s-b_1-b_2-t,
 \qquad
 C:s-c_1-c_2-c_3-t.
\]
If a path system contains the six prescribed paths
\begin{align*}
P_{a_1,b_2}&=a_1s b_1b_2,&
P_{a_2,b_1}&=a_2t b_2b_1,\\
P_{a_1,c_3}&=a_1a_2t c_3,&
P_{a_2,c_2}&=a_2a_1s c_1c_2,\\
P_{b_1,c_3}&=b_1s c_1c_2c_3,&
P_{b_2,c_1}&=b_2t c_3c_2c_1,
\end{align*}
then no positive edge weighting can induce all six paths.
\end{proposition}

\begin{proof}
Write
\begin{align*}
&A_1=w(sa_1),\quad A_2=w(a_1a_2),\quad A_3=w(a_2t),\\
&B_1=w(sb_1),\quad B_2=w(b_1b_2),\quad B_3=w(b_2t),\\
&C_1=w(sc_1),\quad C_2=w(c_1c_2),\quad C_3=w(c_2c_3),\quad C_4=w(c_3t).
\end{align*}
If the six displayed paths are geodesics, then comparison with the natural alternative on the same endpoints gives
\begin{align*}
A_1+B_1+B_2 &\le A_2+A_3+B_3,\\
A_3+B_3+B_2 &\le A_2+A_1+B_1,\\
A_2+A_3+C_4 &\le A_1+C_1+C_2+C_3,\\
A_2+A_1+C_1+C_2 &\le A_3+C_4+C_3,\\
B_1+C_1+C_2+C_3 &\le B_2+B_3+C_4,\\
B_3+C_4+C_3+C_2 &\le B_2+B_1+C_1.
\end{align*}
Adding these inequalities and cancelling common terms yields $2C_2\le0$, impossible for positive edge weights.
\end{proof}

\section{Cycle cuts and Ferrers representations}

This section contains the two auxiliary realization tools used in the positive part of the classification.  The cycle-cut lemma handles the case in which the length-two arm is not the prescribed $s$--$t$ path.  The Ferrers lemmas handle the remaining case $P_{s,t}=s-a-t$ by converting monotone choices of paths into threshold inequalities.

We recall precisely the contraction terminology needed below.  For a vertex $v$, let $T_v$ be the union of all prescribed paths of $\cQ$ having $v$ as an endpoint.  Consistency makes $T_v$ a spanning tree; since $H$ is a cycle, $T_v=H-f(v)$ for a unique edge $f(v)$.  An edge is \emph{$\cQ$-persistent} if it belongs to $T_v$ for every vertex $v$.  Persistent edges may be contracted successively, and the induced quotient path system is consistent.  If $F$ is the set of all persistent edges of a non-trivial path system on a cycle, then $H/F$ is an odd cycle and $\cQ/F$ is its standard shorter-arc system \cite[Section~4.1 and Proposition~4.8]{CizmaLinial2022}.  Conversely, a strict realization of a quotient by a persistent edge expands to a strict realization before contraction \cite[Lemma~4.6]{CizmaLinial2022}.  A \emph{persistent component} is a maximal path of consecutive persistent edges.  These facts are applied only to cycles, where every component and every quotient is unambiguous.

\begin{lemma}\label{lem:persistent_expansion}
Let $H$ be a cycle with distinct vertices $s,t$, let $\cQ$ be a non-trivial consistent path system on $H$, and let $S\subseteq V(H)$ satisfy the following compatibility conditions:
\begin{enumerate}[label=\textup{(\alph*)}]
\item $s\in S$ and $t\notin S$;
\item on each of the two $s$--$t$ arcs of $H$, the set $S$ is a prefix starting at $s$;
\item if $v\notin S$, then the prescribed $t$--$v$ path in $\cQ$ contains no vertex of $S$.
\end{enumerate}
Then, for the purpose of constructing a weighting that both induces $\cQ$ and strictly separates $S$ from its complement by a threshold in the values
\[
        d(s,v)-d(t,v),
\]
the Cizma--Linial contraction of persistent components may be performed in a cut-respecting way.  More precisely, because $S$ is a prefix on each of the two $s$--$t$ arcs, there are at most two cycle edges with one endpoint in $S$ and the other in $V(H)\setminus S$.  Contract persistent components as usual except at these boundary edges: whenever a persistent component contains such a boundary edge $uv$, contract the maximal persistent subpath on the $S$ side of $uv$ and the maximal persistent subpath on the complementary side separately, and retain the edge $uv$ between the two contracted vertices.  Components not meeting a boundary edge are contracted in the ordinary way.  The resulting quotient is a cycle which, after suppressing at most two retained boundary subdivision edges, is the standard odd-cycle quotient of Cizma--Linial.  Its path system is the corresponding subdivision of the standard odd-cycle system.

Suppose this cut-respecting quotient has a positive edge weighting $w^\sharp$ which strictly induces its quotient path system and, for some real number $h$, satisfies
\[
        d_{w^\sharp}(s,v)-d_{w^\sharp}(t,v)<h
        \quad (v\in S^\sharp),
\]
\[
        d_{w^\sharp}(s,v)-d_{w^\sharp}(t,v)>h
        \quad (v\notin S^\sharp),
\]
where $S^\sharp$ is the image of $S$ in the cut-respecting quotient.  Then $w^\sharp$ expands to a positive edge weighting $w$ of the original cycle $H$ which strictly induces $\cQ$ and satisfies the same two strict threshold inequalities for all original vertices.
\end{lemma}

\begin{proof}
The first assertions are consequences of the prefix assumptions.  The set $S$ has at most one boundary edge on each of the two $s$--$t$ arcs, hence at most two boundary edges in the whole cycle.  Splitting a persistent component only at these boundary edges leaves every contracted piece on a single side of the cut and keeps each boundary edge as an explicit edge of the quotient.  Suppressing the retained boundary edges recovers exactly the usual persistent-edge quotient of the cycle system; before suppression, the quotient is just a subdivision of that odd-cycle quotient at at most two places.

It remains to justify the expansion step.  The ordinary persistent-edge expansion lemma of Cizma and Linial \cite[Lemma 4.6]{CizmaLinial2022} says that contracted persistent components may be replaced by sufficiently small positive edge weights without changing the induced cycle system.  We need only record that the same small expansion can be chosen to preserve the cut inequalities.

Indeed, on the cut-respecting quotient all relevant inequalities are strict and finite in number.  They are of two types:
\begin{enumerate}[label=\textup{(\roman*)}]
\item for every pair of quotient vertices $x,y$ and every non-prescribed simple $x$--$y$ path $R$,
\[
        w^\sharp(Q_{x,y}^\sharp)<w^\sharp(R);
\]
\item for every quotient vertex $v$,
\[
        d_{w^\sharp}(s,v)-d_{w^\sharp}(t,v)<h
        \quad\text{if }v\in S^\sharp,
\]
\[
        d_{w^\sharp}(s,v)-d_{w^\sharp}(t,v)>h
        \quad\text{if }v\notin S^\sharp.
\]
\end{enumerate}
Let $\gamma>0$ be smaller than the minimum slack in these inequalities.  Expand every contracted component by assigning positive edge weights whose total weight is less than $\gamma/(10|V(H)|)$.  The sum of all inserted weights is then below $\gamma/10$.  Consequently every simple path length changes by less than $\gamma/10$, every distance changes by less than $\gamma/10$, and every distance difference changes by less than $\gamma/5$.  Hence the strict inequalities involving quotient vertices persist.

For vertices internal to an expanded component contained wholly in one side of the cut, the value of $d(s,v)-d(t,v)$ differs from the value at the contracted quotient vertex by less than $\gamma/5$, so the vertex remains on the same side of the threshold.  For a crossing persistent component, the boundary edge was retained in the quotient and already separates an $S$-side endpoint from a complementary endpoint.  Only the two one-sided pieces adjacent to that boundary edge are expanded, and the same smallness estimate keeps all internal vertices of those pieces on their prescribed sides.  Thus both the strict geodesic inequalities and the cut inequalities survive the expansion.  The resulting positive weighting of $H$ strictly induces $\cQ$ and has the required strict threshold separation.  If $w^\sharp$ and $h$ are rational, the sufficiently small inserted weights may be chosen rational as well.
\end{proof}

\begin{lemma}\label{lem:cycle_cut}
Let $H$ be a cycle with two distinct vertices $s,t$, and let $\cQ$ be a neighborly consistent path system on $H$.  Suppose that a new vertex $a$ is joined to $s$ and $t$, and that $\cQ$ is the restriction of a consistent extension with $P_{a,s}=as$ and $P_{a,t}=at$.  Assume further that each $a$--$v$ path, $v\in V(H)$, begins either with $as$ or with $at$.  Then there is a positive rational edge weighting $w_H$ strictly inducing $\cQ$ and a rational number $h$ such that
\[
 P_{a,v}\text{ begins with }as
 \quad\Longrightarrow\quad
 d_{w_H}(s,v)-d_{w_H}(t,v)<h,
\]
\[
 P_{a,v}\text{ begins with }at
 \quad\Longrightarrow\quad
 d_{w_H}(s,v)-d_{w_H}(t,v)>h.
\]
\end{lemma}

\begin{proof}
Let
\[
        S=\{v\in V(H):P_{a,v}\text{ begins with }as\}.
\]
Consistency of the extension gives the precise cut conditions needed below.  First, $s\in S$ and $t\notin S$.  If $v\in S$, then
\[
        P_{a,v}=as\,Q_{s,v},
\]
and every vertex $u$ of $Q_{s,v}$ also lies in $S$, because the $a$--$u$ subpath of $P_{a,v}$ is $P_{a,u}$ and begins with $as$.  Thus the prescribed $s$--$v$ path stays in $S$.  Symmetrically, if $v\notin S$, then $P_{a,v}=at\,Q_{t,v}$ and the prescribed $t$--$v$ path stays in $V(H)\setminus S$.  In particular, both $H[S]$ and $H[V(H)\setminus S]$ are connected.  Since they are complementary nonempty connected vertex sets of a cycle, each is a cyclic interval.  Equivalently, on each of the two $s$--$t$ arcs of $H$, the set $S$ is a prefix starting at $s$, and the cut has exactly two boundary edges.

Form the cut-respecting quotient $H^\sharp$ described in
Lemma~\ref{lem:persistent_expansion}, and let
\[
        \bar\pi:H^\sharp\longrightarrow C_{2n+1}
\]
suppress its retained boundary edges.  The resulting cycle carries the
standard shorter-arc system.  We first record two points that are essential
when a boundary edge is persistent.

If $xy$ is persistent, then $T_x=T_y$.  Indeed, for every root $z$, the
edge $xy$ lies in the tree $T_z$, and the two paths $Q_{z,x}$ and $Q_{z,y}$
therefore differ in exactly the edge $xy$.  Taking their unions over $z$,
and using the neighborly path $Q_{x,y}=xy$, gives $T_x=T_y$.  Thus the
omitted-edge map $f$ is constant on every persistent component.

Moreover, the two boundary edges cannot belong to the same persistent
component.  Suppose that one persistent component joined both boundaries
through the $S$-arc.  That boundary-to-boundary chain lies in $T_t$.  The
two prescribed paths from $t$ to the complementary-side boundary endpoints
stay outside $S$ and together contain the complementary arc through $t$.
Their union with the persistent chain would contain the whole cycle,
contrary to the fact that $T_t$ is a tree.  If the two complementary-side
endpoints coincide, the persistent component itself is already the whole
cycle, giving the same contradiction.  The case of a component joining the
boundaries through the complementary arc is symmetric, using $T_s$.
Consequently each retained boundary splits a different quotient vertex,
and no quotient vertex has three copies in $H^\sharp$.

\medskip
\noindent\emph{Coincident root images.}
Suppose first that $\bar\pi(s)=\bar\pi(t)=b$.  The persistent component
containing $s$ and $t$ crosses the cut and hence contains a retained
boundary edge $g$.  By the preceding paragraph it contains only that one
boundary.  After its two one-sided pieces are contracted, $s=b_S$ and
$t=b_C$ are the two ends of $g$.  Constancy of $f$ on this component gives
\[
        f(s)=f(t)=e,
\]
where $e$ is the quotient edge antipodal to $b$ in the standard odd-cycle
system.  Hence
\[
        T_s=T_t=H^\sharp-e.
\]
The forest $T_s-g$ has two components, say $K_s$ containing $s$ and $K_t$
containing $t$.  The root-staying property proved above implies
\[
        S^\sharp=V(K_s),
        \qquad
        V(H^\sharp)\setminus S^\sharp=V(K_t);
\]
otherwise a prescribed path rooted at $s$ or at $t$ would cross to the
wrong side.  In particular, the other cut-boundary edge is the
nonpersistent edge $e$.

Give every edge of the suppressed odd cycle weight $1$ and give $g$ weight
$\eta=1/4$.  If two endpoints have distinct images under $\bar\pi$, their
prescribed lift contains at most $n$ quotient edges and at most $g$, while
the complementary arc contains at least $n+1$ quotient edges.  If their
images coincide, the prescribed path is the retained edge $g$ and the
other arc contains all $2n+1$ quotient edges.  Thus this weighting
$w^\sharp$ strictly induces the path system on $H^\sharp$.  Furthermore,
for
\[
        \phi(x)=d_{w^\sharp}(s,x)-d_{w^\sharp}(t,x),
\]
the identities $Q_{t,x}=gQ_{s,x}$ for $x\in S^\sharp$ and
$Q_{s,x}=gQ_{t,x}$ for $x\notin S^\sharp$ give
\[
        \phi(x)=-\eta\quad(x\in S^\sharp),
        \qquad
        \phi(x)=\eta\quad(x\notin S^\sharp).
\]
Taking $h=0$ and applying Lemma~\ref{lem:persistent_expansion} proves the
result in this case.

\medskip
\noindent\emph{Distinct root images.}
Assume now that $\bar\pi(s)$ and $\bar\pi(t)$ are distinct.  Label the
suppressed odd cycle clockwise by $0,1,\ldots,2n$, put
\[
        \bar\pi(s)=0,\qquad \bar\pi(t)=d,\qquad 1\le d\le n,
\]
and set $N=2n+1$ and $e_i=i(i+1)$, with indices modulo $N$.
There are two possible types for the boundary on the short $0$--$d$ arc:
\[
 E_q:\ q\in S^\sharp,\ q+1\notin S^\sharp,\quad 0\le q<d,
\]
or
\[
 V_b:\ b_S-g_b-b_C,\quad 0\le b\le d,
\]
where $g_b$ is retained, $b_S\in S^\sharp$, and
$b_C\notin S^\sharp$.  The endpoint convention is $0_S=s$ and $d_C=t$.
In clockwise order on the long arc, the two types are
\[
 E_p:\ p\notin S^\sharp,\ p+1\in S^\sharp,\quad n\le p\le n+d,
\]
or
\[
 V_c:\ c_C-g_c-c_S,\quad n+1\le c\le n+d.
\]
These ranges follow directly from the root-staying property.  On the short
arc, the selected $0$-rooted and $d$-rooted paths remain on their respective
sides precisely for $0\le b\le d$.  On the long arc, the selected
$0$-rooted path reaches the $S$-copy counterclockwise only when
$c\ge n+1$, while the selected $d$-rooted path reaches the complementary
copy clockwise only when $c\le n+d$.  The same calculation gives the
ranges for $q$ and $p$.  In particular, a split incident with a root can
occur only on the short arc.

For $k\in\mathbb Z/N\mathbb Z$, let $R_k$ assign weight $1$ to
$e_k,e_{k+n}$ and weight $0$ to every other edge.  If the quotient vertex
$b$ is split, define the split ray $U_b$ by
\[
        U_b(g_b)=U_b(e_{b+n})=1,
        \qquad
        U_b(e)=0\quad\text{on every other edge}.
\]
Each $R_k$ non-strictly induces the lifted standard system.  The same is
true of $U_b$: a selected lifted standard arc cannot contain both of its
unit marks, because after suppressing $g_b$ such an arc would contain at
least $n+1$ quotient edges, whereas a selected standard arc contains at
most $n$.  Its complementary arc contains at least one mark.  If the two
endpoints suppress to the same quotient vertex, the selected internal
split segment contains at most the retained mark and its complement
contains the antipodal quotient mark.  This also covers $b=0,d$.

Choose the two boundary rays
\[
 B_s=\begin{cases}R_q,&E_q,\\ U_b,&V_b,\end{cases}
 \qquad
 B_\ell=\begin{cases}R_p,&E_p,\\ U_c,&V_c,\end{cases}
\]
and set
\[
        B=B_s+B_\ell,\qquad Z=R_0,\qquad
        w_0=B+\frac12 Z.
\]
Thus $w_0$ is nonnegative and non-strictly induces the lifted standard
system.  Let $x^-\in S^\sharp,x^+\notin S^\sharp$ be the endpoints of the
short boundary, and let $y^-\in S^\sharp,y^+\notin S^\sharp$ be the
endpoints of the long boundary.  For any weighting $w$ that non-strictly
induces the system, evaluate distances on the prescribed lifts and put
\[
        \phi_w(v)=d_w(s,v)-d_w(t,v),
\]
\[
\begin{aligned}
 F_1(w)&=\phi_w(y^+)-\phi_w(x^-),&
 F_2(w)&=\phi_w(x^+)-\phi_w(y^-),\\
 L_s(w)&=\phi_w(x^+)-\phi_w(x^-),&
 L_\ell(w)&=\phi_w(y^+)-\phi_w(y^-).
\end{aligned}
\]
The complete boundary-incidence calculation is
\[
\begin{array}{c|c|c|c|c}
\text{short/long}&F_1(B)&F_1(Z)&F_2(B)&F_2(Z)\\ \hline
E/E&2&\{-2,0,2\}&\{0,2,4\}&2\\
V/E&2&\{-2,0,2\}&\{0,2,4\}&\{0,2\}\\
E/V&2&\{-2,0\}&\{2,4\}&2\\
V/V&2&\{-2,0\}&2&\{0,2\}.
\end{array}
\]
The two same-boundary gaps at $w_0$ are
\[
\begin{array}{c|c|c}
\text{short/long}&L_s(w_0)&L_\ell(w_0)\\ \hline
E/E&\{2,3,4,5\}&\{1,2,3,4\}\\
V/E&2&\{1,2,3,4\}\\
E/V&\{2,3,4,5\}&2\\
V/V&2&2.
\end{array}
\]
In the $V/E$ row the two set-valued entries must be read jointly.  If the
short boundary is $V_b$ and the long boundary is $E_p$, then
\[
 F_2(Z)=\begin{cases}0,&b=0,\\2,&b>0,\end{cases}
\]
whereas
\[
 F_2(B)=
 \begin{cases}
 2,&b=0,\\
 0,&b>0\text{ and }p=n,\\
 4,&b>0\text{ and }p=n+b,\\
 2,&\text{otherwise}.
 \end{cases}
\]
Thus a zero in one column is always coupled to a positive entry in the
other.

For completeness, we verify the tables rather than appealing to a picture.
For a quotient edge $e_k$, let
\[
 \chi_k(v)=
 \mathbf 1_{\{e_k\in Q^\sharp_{s,v}\}}
 -\mathbf 1_{\{e_k\in Q^\sharp_{t,v}\}}.
\]
At every quotient vertex that can occur as a boundary endpoint,
\[
\chi_k(v)=
\begin{cases}
\mathbf1_{\{0\le k<v\}}-\mathbf1_{\{v\le k<d\}},
   &0\le v\le d,\\
\mathbf1_{\{0\le k<d\}},&v=n,\\
\mathbf1_{\{v\le k\le2n\}}-\mathbf1_{\{d\le k<v\}},
   &n+1\le v\le n+d,\\
-\mathbf1_{\{0\le k<d\}},&v=n+d+1.
\end{cases}
\]
If the last vertex is $N$, it is read as $0$.  The contribution of $R_j$
to $\phi(v)$ is $\chi_j(v)+\chi_{j+n}(v)$.  The contribution of $U_b$ is
$\chi_{b+n}(v)$ plus
\[
 \mathbf1_{\{g_b\in Q^\sharp_{s,v}\}}
 -\mathbf1_{\{g_b\in Q^\sharp_{t,v}\}}.
\]
At the two copies of its own split vertex, the latter term is $-1$ on the
$S$-copy and $+1$ on the complementary copy.  Substituting
\[
 0\le q<d,\qquad 0\le b\le d,\qquad
 n\le p\le n+d,\qquad n+1\le c\le n+d
\]
first for the endpoints of each boundary and then for the crossed endpoint
pairs gives the two displayed tables.  Retaining the alternatives $b=0$,
$p=n$, and $p=n+b$ gives the stated joint $V/E$ values.  The calculation
also covers the smallest case $n=d=1$.

It follows that
\[
        F_1(w_0)\ge1,\qquad F_2(w_0)\ge1,\qquad
        L_s(w_0)\ge1,\qquad L_\ell(w_0)\ge1.
\]
The potential $\phi_{w_0}$ is weakly increasing from $s$ to $t$ along
either $s$--$t$ arc.  At an ordinary central step the $s$-rooted path gains
the traversed edge while the $t$-rooted path loses it; at an outer step both
rooted paths change by the same edge.  At either antipodal switch, the
triangle inequality gives the same nonnegative increment.  A retained
boundary step has positive increment by the last two inequalities above.
Consequently
\[
\begin{aligned}
 \max_{v\in S^\sharp}\phi_{w_0}(v)
 &=\max\{\phi_{w_0}(x^-),\phi_{w_0}(y^-)\}\\
 &<\min\{\phi_{w_0}(x^+),\phi_{w_0}(y^+)\}\\
 &=\min_{v\notin S^\sharp}\phi_{w_0}(v).
\end{aligned}
\]

It remains to make every edge weight positive and every prescribed arc
unique.  Let $u$ assign weight $1$ to each edge of the suppressed odd cycle
and weight $\eta=1/4$ to each retained boundary edge.  For endpoints with
distinct suppressed images, a prescribed lift has at most $n$ quotient
edges and at most two retained edges, whereas its complement has at least
$n+1$ quotient edges.  Hence
\[
        u(Q^\sharp)\le n+2\eta<n+1\le u(\overline{Q^\sharp}).
\]
For endpoints with the same suppressed image, the selected internal split
segment has weight at most $2\eta$, while the complementary arc contains
all $N$ quotient edges.  Thus $u$ is positive and strictly induces the
lifted standard system.  Since $w_0$ induces the same system with ties
allowed,
\[
        w^\sharp=w_0+\varepsilon u
\]
is positive and strict for every $\varepsilon>0$.  Choose a sufficiently
small positive rational $\varepsilon$ so that the strict cut separation
above persists, and choose a rational $h$ between its two sides.  Lemma~
\ref{lem:persistent_expansion} now expands the remaining one-sided
persistent pieces and yields the required $w_H$ and $h$.
\end{proof}

\begin{proposition}\label{prop:theta2_nonshort}
Let $G=\Theta_{2,b,c}$, and let its length-two arm be $A=s-a-t$.  If $\cP$ is neighborly and $P_{s,t}\ne A$, then $\cP$ is strictly metrizable.
\end{proposition}

\begin{proof}
Let $H=G-a$.  Then $H$ is the cycle formed by the two long arms.  No prescribed path whose endpoints lie in $H$ can pass through $a$, because such a path would contain $A$ as its $s$--$t$ subpath, forcing $P_{s,t}=A$ by consistency.  Therefore $\cP$ restricts to a cycle path system $\cQ$ on $H$.

Because $\cP$ is neighborly, the restricted cycle system $\cQ$ is neighborly and hence non-trivial.  Apply Lemma~\ref{lem:cycle_cut} to obtain $w_H$ and $h$.  Choose $M>|h|$ so large that no shortest path between two vertices of $H$ can improve by going through $a$, and set
\[
        w(as)=M,
        \qquad
        w(at)=M+h.
\]
Both weights are positive.  For $v\in V(H)$, the route from $a$ through $s$ is no longer than the route through $t$ exactly when
\[
        d_{w_H}(s,v)-d_{w_H}(t,v)\le h.
\]
The cut inequalities are strict, so every prescribed $a$--$v$ path is the unique shortest one of the two possible routes through $s$ and $t$.  Taking $M$ larger than the total $w_H$-weight of $H$ also makes every route through $a$ strictly longer than the prescribed path between two vertices of $H$.  Hence the resulting weighting strictly induces $\cP$.
\end{proof}

We now spell out the Ferrers terminology used in the proof of $\Theta_{2,b,c}$.  This is the threshold viewpoint on Ferrers or chain graphs, also known in an equivalent weighted form as difference graphs; see, for example, \cite{EhrenborgVanWilligenburg2004,HammerPeledSun1990}.

\begin{definition}\label{def:ferrers}
A \emph{chain} is a finite totally ordered set.  For the standard chain $I=\{1,\ldots,m\}$, a prefix is an initial interval $\{1,\ldots,p\}$, possibly empty or all of $I$; a suffix is a final interval $\{q,\ldots,m\}$, possibly empty or all of $I$.

Let $I$ and $J$ be chains.  A relation $F\subseteq I\times J$ is an \emph{additive Ferrers relation} if it is downward closed in both coordinates: whenever $(i,j)\in F$, $i'\le i$, and $j'\le j$, then $(i',j')\in F$.  Equivalently, the $0$--$1$ matrix of $F$ has left-justified rows and the row lengths are weakly decreasing as $i$ increases.  An \emph{additive threshold representation} of $F$ is a choice of real labels $x_i$ for $i\in I$ and $y_j$ for $j\in J$ such that membership in $F$ is determined by the sign of $x_i+y_j$.

For a single ordered arm $D:s=d_0,d_1,\ldots,d_m=t$, a \emph{potential} is simply a real label $X_i$ assigned to $d_i$.  In the construction below the potentials are strictly increasing along $D$, and the edge length of $d_{i-1}d_i$ is set to $(X_i-X_{i-1})/2$.  This choice is useful because the total length of $D$ is then $(X_m-X_0)/2$, and, writing $d_D$ for the path-length distance restricted to the arm $D$,
\[
        d_D(s,d_i)-d_D(t,d_i)=X_i
\]
when $X_0=-X_m$.  A relation $F\subseteq D^-\times D^+$, where $D^-$ is a prefix and $D^+$ is a suffix of $D$, is a \emph{difference Ferrers relation} if it is downward closed in the prefix coordinate and upward closed in the suffix coordinate:
\[
        (d_i,d_j)\in F,
\quad i'\le i,\quad j'\ge j
        \quad\Longrightarrow\quad
        (d_{i'},d_{j'})\in F.
\]
After reversing the order on $D^+$, this is exactly an additive Ferrers relation.  A \emph{difference threshold representation} realizes membership in $F$ by an inequality of the form $X_j-X_i>\text{constant}$.  A \emph{prefix cut represented by a threshold} means that a prescribed prefix is exactly the set of chain elements whose potential is below that threshold.
\end{definition}

\begin{lemma}\label{lem:additive}
Let $I=\{1,\ldots,m\}$ and $J=\{1,\ldots,n\}$ be chains.  Let $F\subseteq I\times J$ satisfy
\[
(i,j)\in F,
\quad i'\le i,
\quad j'\le j
\quad\Longrightarrow\quad
(i',j')\in F.
\]
Then there are strictly increasing real numbers $x_i$ and $y_j$ such that
\[
        (i,j)\in F
        \quad\Longleftrightarrow\quad
        x_i+y_j<0.
\]
Moreover, if a prefix cut is prescribed in $I$ and a prefix cut is prescribed in $J$, the representation may be chosen so that both cuts are represented by a common threshold $h$:
\[
        i\text{ lies in the prescribed prefix of }I \Longleftrightarrow x_i<h,
\]
\[
        j\text{ lies in the prescribed prefix of }J \Longleftrightarrow y_j<h.
\]
\end{lemma}

\begin{proof}
For each row $i$, let $r_i$ be the largest column index $j$ such that $(i,j)\in F$, and set $r_i=0$ if the $i$th row is empty.  The Ferrers condition says exactly that the row lengths are nested:
\[
        r_1\ge r_2\ge\cdots\ge r_m.
\]
Put $x_i=i$.  For $j\in J$, let
\[
        k_j=\max\{i:r_i\ge j\},
\]
with $k_j=0$ if the set is empty.  Since the row lengths are non-increasing, the sequence $k_j$ is non-increasing in $j$.  Choose $0<\delta<1/(2n)$ and set
\[
        y_j=-k_j-\frac12+\delta j.
\]
Then $y_1<y_2<\cdots<y_n$.  Moreover,
\[
        x_i+y_j=i-k_j-\frac12+\delta j.
\]
If $i\le k_j$, then $x_i+y_j<0$; if $i>k_j$, then $x_i+y_j>0$.  Hence $x_i+y_j<0$ holds exactly when $i\le k_j$, equivalently when $j\le r_i$, which is precisely $(i,j)\in F$.

Now suppose the prescribed prefixes are $I_0=\{1,\ldots,p\}$ and $J_0=\{1,\ldots,q\}$, allowing $p=0,m$ and $q=0,n$.  Starting with the representation just constructed, replace
\[
        x_i\mapsto \lambda x_i+\alpha,
        \qquad
        y_j\mapsto \lambda y_j-\alpha,
\]
where $\lambda>0$.  This multiplies every sum $x_i+y_j$ by $\lambda$, so it preserves all signs and therefore preserves $F$.  A common threshold $h$ must lie in both intervals
\[
        (\lambda x_p+\alpha,\lambda x_{p+1}+\alpha),
        \qquad
        (\lambda y_q-\alpha,\lambda y_{q+1}-\alpha),
\]
where $x_0=y_0=-\infty$ and $x_{m+1}=y_{n+1}=+\infty$.  In the non-extreme case, these intervals overlap whenever
\[
        \frac{\lambda(y_q-x_{p+1})}{2}<\alpha<
        \frac{\lambda(y_{q+1}-x_p)}{2},
\]
and the interval for $\alpha$ is non-empty because $x_p<x_{p+1}$ and $y_q<y_{q+1}$.  The same argument covers extreme prefixes if a missing endpoint is interpreted in the extended real line.  Explicitly, an empty prefix imposes $h<\min x_i$ (or $h<\min y_j$), while a full prefix imposes $h>\max x_i$ (or $h>\max y_j$).  The free parameter $\alpha$ moves the two label intervals in opposite directions, so any combination of these one-sided conditions can be made to overlap.  Choosing $h$ in the resulting non-empty open intersection represents both cuts strictly.
\end{proof}

\begin{lemma}\label{lem:difference}
Let $D:s=d_0,d_1,\ldots,d_m=t$ be a chain.  Let $D^-$ be a prefix containing $s$, let $D^+$ be a suffix containing $t$, assume $D^-\cap D^+=\emptyset$, and let $D^0=D\setminus(D^-\cup D^+)$.  Suppose $F\subseteq D^-\times D^+$ satisfies
\[
(d_i,d_j)\in F,
\quad i'\le i,
\quad j'\ge j
\quad\Longrightarrow\quad
(d_{i'},d_{j'})\in F,
\]
and assume that the row $s$ and the column $t$ belong entirely to $F$.  Fix $L>1$.  Given any strictly increasing prescribed potentials in $(-1,1)$ on $D^0$, there are $M>L$ and a strictly increasing sequence
\[
        -M=X_0<X_1<\cdots<X_m=M
\]
which extends those middle potentials, has $X_i<-L$ on $D^-\setminus\{s\}$ and $X_i>L$ on $D^+\setminus\{t\}$, and satisfies
\[
        (d_i,d_j)\in F
        \quad\Longleftrightarrow\quad
        X_j-X_i>M+L
\]
and
\[
        (d_i,d_j)\notin F
        \quad\Longleftrightarrow\quad
        X_j-X_i<M+L
\]
for all $d_i\in D^-$ and $d_j\in D^+$ with $i<j$.  In particular, equality never occurs.
\end{lemma}

\begin{proof}
Reverse the order on $D^+$: in the reversed order, $d_j$ comes before $d_{j'}$ exactly when $j\ge j'$.  With this reversed order on the second coordinate, the displayed hypothesis is the additive Ferrers condition.  Applying Lemma~\ref{lem:additive}, and then multiplying all potentials and the threshold by a positive constant if necessary, gives real numbers $U_i$ on $D^-$, real numbers $V_j$ on $D^+$, and a threshold $R>0$ such that
\[
        (d_i,d_j)\in F
        \quad\Longleftrightarrow\quad
        U_i+V_j<R.
\]
The explicit half-integer separation in Lemma~\ref{lem:additive} also gives $U_i+V_j>R$ for every pair outside $F$.
The $U_i$ are strictly increasing in the natural order on $D^-$, while the $V_j$ are strictly increasing in the reversed order on $D^+$.  Translate both families and the threshold by replacing $U_i$ with $U_i-U_s$, $V_j$ with $V_j-V_t$, and $R$ with $R-U_s-V_t$.  This preserves the displayed equivalence and gives $U_s=0$ and $V_t=0$.  Since the row $s$ and the column $t$ belong entirely to $F$, we have $V_j<R$ for every $d_j\in D^+$ and $U_i<R$ for every $d_i\in D^-$.

Set $M=R+L$.  For the prefix and suffix vertices define
\[
        X_i=-M+U_i\quad(d_i\in D^-),
        \qquad
        X_j=M-V_j\quad(d_j\in D^+).
\]
Then $X_s=-M$ and $X_t=M$.  The values on $D^-$ are strictly increasing because the $U_i$ are strictly increasing.  The values on $D^+$ are also strictly increasing in the natural order along $D$: as $j$ increases, the reversed-order potential $V_j$ strictly decreases, and hence $M-V_j$ strictly increases.  Because $U_i<R$ and $V_j<R$, all non-end prefix values are below $-L$ and all non-end suffix values are above $L$.

Insert the prescribed middle potentials on $D^0$.  These lie in $(-1,1)$, while the prefix values lie below $-L$ and the suffix values lie above $L$ with $L>1$; therefore the full sequence $X_0<X_1<\cdots<X_m$ is strictly increasing.  Finally,
\[
        X_j-X_i=2M-(U_i+V_j),
\]
and, since $M=R+L$,
\[
        X_j-X_i>M+L
        \quad\Longleftrightarrow\quad
        U_i+V_j<R.
\]
For pairs outside $F$ the strict reverse inequality $U_i+V_j>R$ gives $X_j-X_i<M+L$.  Thus equality is excluded and the sequence has all required properties.
\end{proof}

\begin{remark}\label{rem:difference-disjoint}
The disjointness hypothesis in Lemma~\ref{lem:difference} is essential for the stated range conditions.  It is automatically satisfied in the two applications below: there, an internal vertex belonging to both $D^-$ and $D^+$ would force two different prescribed $s$--$d_i$ subpaths, contradicting consistency.
\end{remark}

\begin{lemma}\label{lem:two_threshold}
Let $I$ and $J$ be possibly empty finite chains, each of size at most $2$, and let $F\subseteq I\times J$ be an additive Ferrers relation.  Let
\[
        I_b\subseteq I_a\subseteq I,
        \qquad
        J_b\subseteq J_a\subseteq J
\]
be prescribed prefixes.  Then there are strictly increasing real labels $x_i$ for $i\in I$, strictly increasing real labels $y_j$ for $j\in J$, and two real thresholds $h_b<h_a$ such that
\[
        (i,j)\in F \Longrightarrow x_i+y_j\le0,
        \qquad
        (i,j)\notin F \Longrightarrow x_i+y_j\ge0,
\]
\[
        i\in I_b \Longrightarrow x_i\le h_b,
        \qquad
        i\notin I_b \Longrightarrow x_i\ge h_b,
\]
\[
        i\in I_a \Longrightarrow x_i\le h_a,
        \qquad
        i\notin I_a \Longrightarrow x_i\ge h_a,
\]
and the same two threshold conditions hold for $J_b\subseteq J_a$ with the labels $y_j$.  The labels and thresholds may also be made arbitrarily small by a common positive rescaling.
\end{lemma}

\begin{proof}
Assign to each element of $I$ a type $0,1,$ or $2$: type $0$ means membership in $I_b$, type $1$ means membership in $I_a\setminus I_b$, and type $2$ means membership in $I\setminus I_a$.  Define the types in $J$ similarly.  Since all four sets are prefixes, the type sequence along each chain is non-decreasing.

We first choose auxiliary labels $\xi_i,\eta_j$ and a real number $T$ with the following properties:
\[
        (i,j)\in F \Longrightarrow \xi_i+\eta_j\le T,
        \qquad
        (i,j)\notin F \Longrightarrow \xi_i+\eta_j\ge T,
\]
and each auxiliary label lies in the interval prescribed by its type,
\[
        \text{type }0: (-\infty,-1],
        \qquad
        \text{type }1: [-1,1],
        \qquad
        \text{type }2: [1,\infty).
\]
The auxiliary labels are required to be strictly increasing along their chains.

If one of the chains is empty, there is nothing to separate.  If one chain has a single element, say $I=\{1\}$, then a Ferrers relation in $I\times J$ is simply an initial segment of the chain $J$.  Choose $\xi_1$ in its prescribed type interval and choose the labels $\eta_j$ strictly increasing inside their prescribed type intervals; if necessary, move the labels in the unbounded type-$0$ or type-$2$ intervals farther apart.  The sums $\xi_1+\eta_j$ are then strictly increasing in $j$, so a threshold $T$ separates the initial segment.  The case $|J|=1$ is symmetric.  Suppose therefore that $I=\{1,2\}$ and $J=\{1,2\}$.  Put
\[
        s_{ij}=\xi_i+\eta_j.
\]
Since $\xi_1<\xi_2$ and $\eta_1<\eta_2$, we always have
\[
        s_{11}<s_{12}<s_{22},
        \qquad
        s_{11}<s_{21}<s_{22}.
\]
Thus every Ferrers lower ideal in the $2\times2$ grid is automatically separable by a threshold except possibly the two ideals of size two.  If
\[
        F=\{(1,1),(1,2)\},
\]
then separation is equivalent to $s_{12}\le s_{21}$, that is,
\[
        \eta_2-\eta_1\le \xi_2-\xi_1.
\]
If
\[
        F=\{(1,1),(2,1)\},
\]
then separation is equivalent to $s_{21}\le s_{12}$, that is,
\[
        \xi_2-\xi_1\le \eta_2-\eta_1.
\]
It remains only to know that the two gaps can be chosen with either prescribed weak order while respecting the type intervals.  For a non-decreasing two-term type sequence, the possible positive gaps are as follows: the type sequence $(1,1)$ allows exactly the gaps $0<g\le2$; the type sequence $(0,2)$ allows exactly the gaps $g\ge2$; and the four remaining type sequences allow every positive gap.  Hence, to impose $\eta_2-\eta_1\le \xi_2-\xi_1$, choose the $\xi$-gap at least as large as the $\eta$-gap; in the only extremal case $(\xi\text{-types})=(1,1)$ and $(\eta\text{-types})=(0,2)$, take both gaps equal to $2$.  The opposite inequality is handled symmetrically.  This proves the existence of the auxiliary labels and the separating threshold $T$.

Now set
\[
        H=-\frac{T}{2},
        \qquad
        x_i=H+\xi_i,
        \qquad
        y_j=H+\eta_j,
\]
and choose
\[
        h_b=H-1,
        \qquad
        h_a=H+1.
\]
Then $h_b<h_a$, and the type intervals above translate exactly into the two required prefix-threshold conditions.  Moreover
\[
        x_i+y_j\le0
        \quad\Longleftrightarrow\quad
        \xi_i+\eta_j\le T,
\]
and similarly for the reversed weak inequality.  A common positive rescaling preserves all weak inequalities, all represented cuts, and the strict inequality $h_b<h_a$.  The lemma follows.
\end{proof}

\begin{lemma}[four-route comparison]\label{lem:four-route}
Let $B$ and $C$ be two $s$--$t$ arms of total weights $M_B$ and $M_C$, and let a third $s$--$t$ arm have total weight $L$.  Normalize potentials on $B$ and $C$ so that
\[
 d_B(s,b_i)-d_B(t,b_i)=X_i,
 \qquad
 d_C(s,c_j)-d_C(t,c_j)=Y_j.
\]
For internal vertices $b_i,c_j$, the four simple $b_i$--$c_j$ routes are
\[
\begin{array}{ll}
 S=b_iBsCc_j, & T=b_iBtCc_j,\\
 U=b_iBsAtCc_j, & V=b_iBtAsCc_j,
\end{array}
\]
where $A$ denotes the third arm.  The following conditions are necessary and sufficient for the indicated route to be a geodesic:
\[
\begin{array}{c|ccc}
\text{route} & \multicolumn{3}{c}{\text{conditions}}\\ \hline
S&X_i\le L&Y_j\le L&X_i+Y_j\le0\\
T&X_i\ge-L&Y_j\ge-L&X_i+Y_j\ge0\\
U&X_i\le-L&Y_j\ge L&\\
V&X_i\ge L&Y_j\le-L&
\end{array}
\]
If every displayed inequality in the relevant row is strict, that route is the unique geodesic.
\end{lemma}

\begin{proof}
After cancelling the common term $(M_B+M_C)/2$, the four route lengths are respectively
\[
 \frac{X_i+Y_j}{2},\qquad
 -\frac{X_i+Y_j}{2},\qquad
 L+\frac{X_i-Y_j}{2},\qquad
 L+\frac{-X_i+Y_j}{2}.
\]
Pairwise comparison gives exactly the table.  Strict inequalities make the selected value strictly smaller than each of the other three.
\end{proof}

\section{Metrizable cases and proof of the classification}

We now prove the two positive cases not already covered by outerplanarity, and then complete the proof of Theorem~\ref{thm:main}.  The first case is the infinite family $\Theta_{2,b,c}$; the second is the exceptional graph $\Theta_{3,3,3}$.

\begin{theorem}\label{thm:theta2bc}
For every $2\le b\le c$, the theta graph $\Theta_{2,b,c}$ is strictly metrizable.
\end{theorem}

\begin{proof}
Let the length-two arm be
\[
        A=s-a-t,
\]
and write the other arms as
\[
        B:s=b_0,b_1,\ldots,b_b=t,
        \qquad
        C:s=c_0,c_1,\ldots,c_c=t.
\]
Let $\cP$ be a consistent path system.  By Proposition~\ref{prop:nonneighborly}, we may assume $\cP$ is neighborly.  By Proposition~\ref{prop:theta2_nonshort}, we may assume
\[
        P_{s,t}=A=s-a-t.
\]

For a long arm $D:s=d_0,d_1,\ldots,d_m=t$, define three vertex sets
\[
        D^- = \{s\}\cup\{d_i:1\le i\le m-1,\ P_{d_i,t}=d_iDs\,a\,t\},
\]
\[
        D^+ = \{t\}\cup\{d_i:1\le i\le m-1,\ P_{s,d_i}=s\,a\,tDd_i\},
\]
and
\[
        D^0=D\setminus(D^-\cup D^+).
\]
Thus $D^-$ consists of vertices whose prescribed path to $t$ first goes backward along $D$ to $s$ and then crosses the length-two arm $A$; $D^+$ consists of vertices whose prescribed path from $s$ first crosses $A$ and then follows $D$ backward from $t$; and $D^0$ is the middle part where neither of these two forced behaviours occurs.

Consistency implies that $D^-$ is a prefix, $D^+$ is a suffix, and $D^-\cap D^+=\varnothing$.  For example, if $d_i\in D^-$ and $i'\le i$, then the $d_{i'}$--$t$ subpath of $P_{d_i,t}$ is $d_{i'}Ds\,a\,t$, so $d_{i'}\in D^-$.  The assertion for $D^+$ is symmetric.  If an internal vertex $d_i$ belonged to both $D^-$ and $D^+$, then the $s$--$d_i$ subpath of $P_{d_i,t}$ would be the direct subpath $sDd_i$, whereas $P_{s,d_i}$ would be $s\,a\,tDd_i$, contradicting consistency.

\begin{figure}[t]
\centering
\begin{tikzpicture}[x=1cm,y=1cm]
  \draw[very thick] (0,0)--(8,0);
  \foreach \x in {0,1,2,3,4,5,6,7,8} \fill (\x,0) circle (1.4pt);
  \draw[very thick,blue] (0,0)--(2.6,0);
  \draw[very thick,orange] (5.4,0)--(8,0);
  \node[below] at (0,0) {$s$}; \node[below] at (8,0) {$t$};
  \node[above,blue] at (1.3,0) {$D^-$};
  \node[above] at (4,0) {$D^0$};
  \node[above,orange] at (6.7,0) {$D^+$};
  \node[below=7mm] at (1.3,0) {$X<-L$};
  \node[below=7mm] at (4,0) {$-1<X<1$};
  \node[below=7mm] at (6.7,0) {$X>L$};
\end{tikzpicture}
\caption{The consistency partition and potential ranges on a long arm.}
\label{fig:arm-partition}
\end{figure}

For $i<j$ on the same long arm $D$, if $P_{d_i,d_j}$ uses $A$, then $d_i\in D^-$ and $d_j\in D^+$.  The set of such same-arm pairs is a difference Ferrers relation on $D^-\times D^+$: if $(d_i,d_j)$ uses $A$, $i'\le i$, and $j'\ge j$, then $P_{d_i,d_j}$ contains $d_{i'}Ds\,a\,tDd_{j'}$ as a subpath, so $P_{d_{i'},d_{j'}}$ also uses $A$.

Next consider a cross-arm pair $b_i\in V(B)\setminus\{s,t\}$ and $c_j\in V(C)\setminus\{s,t\}$.  There are four relevant simple routes between them:
\[
S:b_iBs\,Cc_j,
\quad
T:b_iBt\,Cc_j,
\quad
U:b_iBs\,a\,tCc_j,
\quad
V:b_iBt\,a\,sCc_j.
\]
The letters record which junctions are used: $S$ goes through the branch vertex $s$, $T$ goes through $t$, $U$ crosses $A$ from $s$ to $t$, and $V$ crosses $A$ from $t$ to $s$.  Consistency gives
\[
        P_{b_i,c_j}=U
        \quad\Longleftrightarrow\quad
        b_i\in B^-\text{ and }c_j\in C^+,
\]
\[
        P_{b_i,c_j}=V
        \quad\Longleftrightarrow\quad
        b_i\in B^+\text{ and }c_j\in C^-.
\]
Indeed, the forward implication follows by taking the $b_i$--$t$ and $s$--$c_j$ subpaths.  Conversely, if $b_i\in B^-$ and $c_j\in C^+$, then the routes $S,T,V$ would give a wrong $b_i$--$t$ or $s$--$c_j$ subpath, so only $U$ is possible.  The proof for $V$ is symmetric.

Outside the two exceptional rectangles $B^-\times C^+$ and $B^+\times C^-$, the route is either $S$ or $T$.  More precisely, away from these rectangles, membership in $B^-$ or $C^-$ forces the route through $s$, while membership in $B^+$ or $C^+$ forces the route through $t$; otherwise one obtains a wrong $b_i$--$t$, $s$--$c_j$, $s$--$b_i$, or $c_j$--$t$ subpath.  On the central rectangle $B^0\times C^0$, the $S$-region is an additive Ferrers relation: if $P_{b_i,c_j}=S$, $i'\le i$, and $j'\le j$, then $P_{b_i,c_j}$ contains $b_{i'}Bs\,Cc_{j'}$ as a subpath, hence $P_{b_{i'},c_{j'}}=S$.

Consistency also controls the paths from $a$ to the long arms.  Vertices in $D^-$ have their $a$--$d_i$ path beginning with $as$, and vertices in $D^+$ have their $a$--$d_i$ path beginning with $at$.  The opposite choice would either give the direct $d_i$--$t$ or $s$--$d_i$ subpath where the definition of $D^-$ or $D^+$ prescribes a path through $A$, or would contain an $s$--$t$ subpath different from $A$.

Fix $L=2$.  The vertices $d_i$ of any long arm whose $a$--$d_i$ path begins with $as$ form a prefix: if $i'\le i$ and $P_{a,d_i}$ begins with $as$, then the $a$--$d_{i'}$ subpath of $P_{a,d_i}$ also begins with $as$.  If $B^0\times C^0$ is empty, choose arbitrary strictly increasing middle potentials in $(-1,1)$ on the non-empty middle chain or chains, and choose a real threshold $h$ representing the relevant prefix cuts there; if both middle chains are empty, take $h=0$.  Since a finite ordered set with a prescribed prefix is separated by an open interval of possible thresholds, after a common small rescaling we may also require $|h|<1$.  If both $B^0$ and $C^0$ are non-empty, apply Lemma~\ref{lem:additive} to the Ferrers $S$-region on $B^0\times C^0$, together with the prefix cuts on $B^0$ and $C^0$ determined by whether $P_{a,b_i}$ and $P_{a,c_j}$ begin with $as$.  This gives strictly increasing middle potentials $X_i$ for $b_i\in B^0$, $Y_j$ for $c_j\in C^0$, and a threshold $h$, such that
\[
P_{a,b_i}\text{ begins with }as \Longleftrightarrow X_i<h,
\qquad
P_{a,c_j}\text{ begins with }as \Longleftrightarrow Y_j<h,
\]
and
\[
P_{b_i,c_j}=S \Longleftrightarrow X_i+Y_j<0
\quad (b_i\in B^0,\ c_j\in C^0).
\]
Scaling all middle potentials and, when necessary, the threshold by a common small positive factor, assume $|X_i|,|Y_j|,|h|<1$ whenever the corresponding quantities are present.

Now apply Lemma~\ref{lem:difference} separately to the same-arm Ferrers relations on $B^-\times B^+$ and $C^-\times C^+$, preserving the already chosen middle potentials.  We obtain strictly increasing potential sequences
\[
        -M_B=X_0<X_1<\cdots<X_b=M_B,
\]
\[
        -M_C=Y_0<Y_1<\cdots<Y_c=M_C,
\]
with prefix values below $-L$, middle values in $(-1,1)$, suffix values above $L$, and with same-arm paths through $A$ represented exactly by
\[
        X_j-X_i>M_B+L
        \quad\text{on }B^-\times B^+,
\]
\[
        Y_j-Y_i>M_C+L
        \quad\text{on }C^-\times C^+.
\]

Define edge weights by
\[
        w(sa)=\frac{L-h}{2},
        \qquad
        w(at)=\frac{L+h}{2},
\]
\[
        w(b_{i-1}b_i)=\frac{X_i-X_{i-1}}2,
        \qquad
        w(c_{j-1}c_j)=\frac{Y_j-Y_{j-1}}2.
\]
All weights are positive because $|h|<L$ and the potential sequences are strictly increasing.  The total lengths of $B$ and $C$ are $M_B$ and $M_C$, respectively, and the potential normalization gives
\[
        d_B(s,b_i)-d_B(t,b_i)=X_i,
        \qquad
        d_C(s,c_j)-d_C(t,c_j)=Y_j.
\]

We verify that every prescribed path is a geodesic.  First consider $a$ and a vertex $d_i$ on a long arm $D$, and let $E$ be the other long arm.  Write $Z_i$ for the corresponding potential and $M_D,M_E$ for the total lengths of $D,E$.  The two direct routes through the endpoints of $A$ have lengths
\[
        R_s=\frac{L-h}{2}+\frac{M_D+Z_i}{2},
        \qquad
        R_t=\frac{L+h}{2}+\frac{M_D-Z_i}{2}.
\]
Thus $R_s\le R_t$ exactly when $Z_i\le h$.  There are also two routes that use the other long arm $E$: their lengths are
\[
        R_s+M_E-Z_i,
        \qquad
        R_t+M_E+Z_i.
\]
If the prescribed path begins with $as$, then either $d_i\in D^-$, so $Z_i<-L$, or $d_i\in D^0$ and the threshold construction gives $Z_i<h<1$.  Hence $R_s\le R_t$ and
\[
        (R_s+M_E-Z_i)-R_s=M_E-Z_i>0,
\]
\[
        (R_t+M_E+Z_i)-R_s=M_E+h>0,
\]
because $M_E>L=2$ and $|h|<1$.  If the prescribed path begins with $at$, then either $d_i\in D^+$, so $Z_i>L$, or $d_i\in D^0$ and $Z_i\ge h>-1$.  Hence $R_t\le R_s$ and
\[
        (R_t+M_E+Z_i)-R_t=M_E+Z_i>0,
\]
\[
        (R_s+M_E-Z_i)-R_t=M_E-h>0.
\]
Therefore every prescribed $a$--$d_i$ path is the unique geodesic.

For two vertices $d_i,d_j$ on the same long arm $D$, $i<j$, let $Z$ denote the corresponding potential sequence and let $M_D$ be the total length of the arm.  The direct route along $D$ has length $(Z_j-Z_i)/2$, while the route through $A$ has length
\[
        M_D+L-\frac{Z_j-Z_i}{2}.
\]
Thus the route through $A$ is no longer than the direct route exactly when $Z_j-Z_i\ge M_D+L$.  If the prescribed path uses $A$, then $(d_i,d_j)\in D^-\times D^+$ and Lemma~\ref{lem:difference} gives $Z_j-Z_i>M_D+L$.  If the prescribed path is direct and $(d_i,d_j)\in D^-\times D^+$, the strict complementary part of that lemma gives $Z_j-Z_i<M_D+L$.  In all remaining direct cases the latter strict inequality follows from the potential ranges: two prefix vertices, two suffix vertices, or two middle vertices differ by less than $M_D+L$, and a prefix-to-middle or middle-to-suffix pair differs by less than $M_D+1<M_D+L$.  The route through the other long arm is strictly longer than the route through $A$, because the other long arm has total length larger than $L$.

Finally take $b_i\in V(B)\setminus\{s,t\}$ and $c_j\in V(C)\setminus\{s,t\}$.  Apply Lemma~\ref{lem:four-route}.  The rectangles $B^-\times C^+$ and $B^+\times C^-$ are exactly the $U$ and $V$ regions, and the strict potential ranges give the strict inequalities in the corresponding rows.  Suppose next that $b_i\in B^-$ and $c_j\notin C^+$.  Then $X_i<-L$, while $Y_j<L$, so all inequalities for $S$ are strict.  The case $c_j\in C^-$ and $b_i\notin B^+$ is symmetric.  Similarly, if $b_i\in B^+$ and $c_j\notin C^-$, then $X_i>L$ and $Y_j>-L$, so all inequalities for $T$ are strict; the remaining symmetric case is identical.  On $B^0\times C^0$, Lemma~\ref{lem:additive} gives the required strict sign of $X_i+Y_j$.  Hence every cross-arm prescribed path is the unique geodesic.  Together with Proposition~\ref{prop:nonneighborly} and Proposition~\ref{prop:theta2_nonshort}, this proves strict metrizability.
\end{proof}

The remaining positive case is the exceptional graph with three arms of length three.  Its proof uses the same potentials as Theorem~\ref{thm:theta2bc}, but the chosen $s$--$t$ arm now has two internal vertices and therefore produces two thresholds instead of one.

\begin{proposition}\label{prop:theta333}
The graph $\Theta_{3,3,3}$ is metrizable.
\end{proposition}

\begin{proof}
Let $\cP$ be a consistent path system on $\Theta_{3,3,3}$.  By Proposition~\ref{prop:nonneighborly}, we may assume that $\cP$ is neighborly.  The prescribed $s$--$t$ path is one of the three length-three arms.  By symmetry, write it as
\[
        A:s-a-b-t.
\]
Write the other two arms as
\[
        B:s=b_0,b_1,b_2,b_3=t,
        \qquad
        C:s=c_0,c_1,c_2,c_3=t.
\]

For $D\in\{B,C\}$, written as $D:s=d_0,d_1,d_2,d_3=t$, define
\[
        D^- = \{s\}\cup\{d_i:1\le i\le2,
        \ P_{d_i,t}=d_iDs\,a\,b\,t\},
\]
\[
        D^+ = \{t\}\cup\{d_i:1\le i\le2,
        \ P_{s,d_i}=s\,a\,b\,tDd_i\},
\]
and set $D^0=D\setminus(D^-\cup D^+)$.  As in the proof of Theorem~\ref{thm:theta2bc}, consistency implies that $D^-$ is a prefix, $D^+$ is a suffix, and $D^-\cap D^+=\varnothing$.  Since $\cP$ is neighborly, $d_2\notin D^-$ and $d_1\notin D^+$; in particular $D^0$ has at most two vertices.

For comparable pairs $d_i\in D^-$ and $d_j\in D^+$ with $i<j$, let $F_D$ be the set of pairs for which $P_{d_i,d_j}$ uses the arm $A$, i.e.
\[
        P_{d_i,d_j}=d_iDs\,a\,b\,tDd_j.
\]
Then $F_D$ is a difference Ferrers relation.  Indeed, if $(d_i,d_j)\in F_D$, $i'\le i$, and $j'\ge j$, then the $d_{i'}$--$d_{j'}$ subpath of $P_{d_i,d_j}$ is $d_{i'}Ds\,a\,b\,tDd_{j'}$, so consistency puts $(d_{i'},d_{j'})$ in $F_D$.  The row $s$ and the column $t$ are contained in $F_D$ by the definitions of $D^+$ and $D^-$ and by $P_{s,t}=A$.

Next consider cross-arm pairs $b_i\in V(B)\setminus\{s,t\}$ and $c_j\in V(C)\setminus\{s,t\}$.  There are four relevant routes:
\[
\begin{aligned}
S&: b_iBsCc_j,        &
T&: b_iBtCc_j,\\
U&: b_iBs\,a\,b\,tCc_j, &
V&: b_iBt\,b\,a\,sCc_j .
\end{aligned}
\]
The same consistency argument as in Theorem~\ref{thm:theta2bc} gives
\[
        P_{b_i,c_j}=U
        \Longleftrightarrow
        b_i\in B^-\text{ and }c_j\in C^+,
\]
\[
        P_{b_i,c_j}=V
        \Longleftrightarrow
        b_i\in B^+\text{ and }c_j\in C^-.
\]
Outside these two rectangles the prescribed route is $S$ or $T$, and on the central rectangle $B^0\times C^0$ the $S$-region is an additive Ferrers relation: if $P_{b_i,c_j}=S$, $i'\le i$, and $j'\le j$, then $P_{b_i,c_j}$ contains $b_{i'}BsCc_{j'}$ as a subpath, so $P_{b_{i'},c_{j'}}=S$.

We now record the two cuts produced by the internal vertices $a$ and $b$ of the arm $A$.  For a vertex $v\in B\cup C$, let
\[
        v\in S_a \quad\Longleftrightarrow\quad
        P_{a,v}\text{ begins with }as,
\]
and
\[
        v\in S_b \quad\Longleftrightarrow\quad
        P_{b,v}\text{ begins with }ba.
\]
Consistency gives
\[
        D^-\subseteq S_b\cap D
        \subseteq S_a\cap D
        \subseteq D\setminus D^+
        \qquad(D=B,C).
\]
For instance, if $v\in D^-$, then $P_{v,t}=vDsabt$ contains the subpaths $vDsa$ and $vDsab$, forcing both $P_{a,v}$ and $P_{b,v}$ to go through $s$.  If $v\in D^+$, the path $P_{s,v}=sabtDv$ forces both $P_{a,v}$ and $P_{b,v}$ to go through $t$.  Finally, $S_b\subseteq S_a$ because a path $P_{b,v}$ beginning with $ba$ contains the $a$--$v$ subpath beginning with $as$.  On each $D^0$, both $S_a\cap D^0$ and $S_b\cap D^0$ are prefixes, again by taking subpaths.

Apply Lemma~\ref{lem:two_threshold} to the additive Ferrers $S$-region on $B^0\times C^0$, with the nested prefixes
\[
        S_b\cap B^0\subseteq S_a\cap B^0,
        \qquad
        S_b\cap C^0\subseteq S_a\cap C^0.
\]
We obtain strictly increasing middle potentials $X_i$ on $B^0$, strictly increasing middle potentials $Y_j$ on $C^0$, and thresholds $h_b<h_a$ such that
\[
        P_{b_i,c_j}=S \Longrightarrow X_i+Y_j\le0,
        \qquad
        P_{b_i,c_j}=T \Longrightarrow X_i+Y_j\ge0
\]
for $b_i\in B^0$, $c_j\in C^0$, and such that the two prefix cuts are represented by the same thresholds:
\[
        v\in S_b \Longrightarrow Z_v\le h_b,
        \qquad
        v\notin S_b \Longrightarrow Z_v\ge h_b,
\]
\[
        v\in S_a \Longrightarrow Z_v\le h_a,
        \qquad
        v\notin S_a \Longrightarrow Z_v\ge h_a,
\]
where $Z_{b_i}=X_i$ and $Z_{c_j}=Y_j$ on the central vertices.  Multiplying all middle potentials and the two thresholds by a sufficiently small positive constant, assume
\[
        |X_i|<1,
        \qquad
        |Y_j|<1,
        \qquad
        |h_b|<1,
        \qquad
        |h_a|<1.
\]

Set $L=2$.  Apply Lemma~\ref{lem:difference} separately to the difference Ferrers relations $F_B$ and $F_C$, preserving the middle potentials just chosen.  We obtain strictly increasing sequences
\[
        -M_B=X_0<X_1<X_2<X_3=M_B,
\]
\[
        -M_C=Y_0<Y_1<Y_2<Y_3=M_C,
\]
where $M_B,M_C>L$, all non-end prefix values are below $-L$, all non-end suffix values are above $L$, and the same-arm paths through $A$ are represented by
\[
        (b_i,b_j)\in F_B
        \Longleftrightarrow
        X_j-X_i>M_B+L,
\]
\[
        (c_i,c_j)\in F_C
        \Longleftrightarrow
        Y_j-Y_i>M_C+L.
\]

Define the weights on the chosen arm $A$ by
\[
        w(sa)=\frac{L-h_a}{2},
        \qquad
        w(ab)=\frac{h_a-h_b}{2},
        \qquad
        w(bt)=\frac{L+h_b}{2}.
\]
These three numbers are positive because $|h_a|,|h_b|<L$ and $h_b<h_a$, and their sum is $L$.  On the two remaining arms define
\[
        w(b_{i-1}b_i)=\frac{X_i-X_{i-1}}{2},
        \qquad
        w(c_{j-1}c_j)=\frac{Y_j-Y_{j-1}}{2}.
\]
All these weights are positive.  The total lengths of $B$ and $C$ are $M_B$ and $M_C$, and
\[
        d_B(s,b_i)-d_B(t,b_i)=X_i,
        \qquad
        d_C(s,c_j)-d_C(t,c_j)=Y_j.
\]

We check that every prescribed path is a geodesic.  First, the subpaths of $A$ are geodesics: any alternative to a proper subpath of $A$ must use one of the other arms, whose total length is larger than $L$, while every proper subpath of $A$ has length smaller than $L$; and $P_{s,t}=A$ has length $L<M_B,M_C$.

Consider $a$ and a vertex $v=d_i$ on $D\in\{B,C\}$.  Let $E$ be the other long arm, and let $Z_i$, $M_D$, and $M_E$ denote the corresponding potential and total arm lengths.  The two direct routes through $s$ and through $t$ have lengths
\[
        R_s^a=\frac{L-h_a}{2}+\frac{M_D+Z_i}{2},
\]
\[
        R_t^a=\frac{h_a-h_b}{2}+\frac{L+h_b}{2}
              +\frac{M_D-Z_i}{2}.
\]
Thus $R_s^a\le R_t^a$ exactly when $Z_i\le h_a$.  If $P_{a,v}$ begins with $as$, the threshold construction, together with the inclusion $D^-\subseteq S_a$, gives $Z_i\le h_a$; if $P_{a,v}$ begins with $ab$, it gives $Z_i\ge h_a$.

It remains only to compare with the two routes that use the other long arm.  The route $a-s-E-t-D-v$ has length $R_s^a+M_E-Z_i$, and the route $a-b-t-E-s-D-v$ has length $R_t^a+M_E+Z_i$.  If $Z_i\le h_a$, then
\[
        M_E-Z_i>0,
        \qquad
        (R_t^a+M_E+Z_i)-R_s^a=M_E+h_a>0.
\]
If $Z_i\ge h_a$, then
\[
        M_E+Z_i>0,
        \qquad
        (R_s^a+M_E-Z_i)-R_t^a=M_E-h_a>0.
\]
Here we use $M_E>L>1$ and $|h_a|<1$, together with the facts that prefix potentials are $<-L$, suffix potentials are $>L$, and middle potentials lie in $(-1,1)$.  Hence every prescribed $a$--$v$ path is a geodesic.

The argument for $b$ is identical, using the lower threshold.  The route from $b$ to $v=d_i$ through $s$ has length
\[
        R_s^b=\frac{h_a-h_b}{2}+\frac{L-h_a}{2}
              +\frac{M_D+Z_i}{2},
\]
whereas the route through $t$ has length
\[
        R_t^b=\frac{L+h_b}{2}+\frac{M_D-Z_i}{2}.
\]
Thus $R_s^b\le R_t^b$ exactly when $Z_i\le h_b$.  The two routes through the other long arm are longer by the same calculation with $h_b$ in place of $h_a$.  Hence every prescribed $b$--$v$ path is a geodesic.

For two vertices $d_i,d_j$ on the same arm $D\in\{B,C\}$, with $i<j$, the direct subpath of $D$ has length $(Z_j-Z_i)/2$, while the route using $A$ has length
\[
        M_D+L-\frac{Z_j-Z_i}{2}.
\]
The route using the other long arm is still longer, since that arm has total length larger than $L$.  Therefore the route through $A$ is geodesic exactly when $Z_j-Z_i\ge M_D+L$, and the direct route is geodesic exactly when $Z_j-Z_i\le M_D+L$.  If the prescribed path uses $A$, then $(d_i,d_j)\in F_D$ and Lemma~\ref{lem:difference} gives $Z_j-Z_i>M_D+L$.  If the prescribed path is direct and $(d_i,d_j)\in D^-\times D^+$, the same lemma gives $Z_j-Z_i\le M_D+L$.  In all other direct cases the inequality $Z_j-Z_i\le M_D+L$ follows from the potential ranges: two prefix vertices, two suffix vertices, or two middle vertices differ by less than $M_D+L$, and a prefix-to-middle or middle-to-suffix pair differs by less than $M_D+1<M_D+L$.

Finally take $b_i\in V(B)\setminus\{s,t\}$ and $c_j\in V(C)\setminus\{s,t\}$.  The necessary and sufficient comparisons are those of Lemma~\ref{lem:four-route}.
The rectangles $B^-\times C^+$ and $B^+\times C^-$ are exactly the $U$ and $V$ regions, and the potential ranges give the corresponding inequalities.  If $b_i\in B^-$ and $c_j\notin C^+$, then $X_i\le -L$ and $Y_j<L$, so $S$ is geodesic; the case $c_j\in C^-$ and $b_i\notin B^+$ is symmetric.  If $b_i\in B^+$ and $c_j\notin C^-$, then $X_i\ge L$ and $Y_j>-L$, so $T$ is geodesic; the case $c_j\in C^+$ and $b_i\notin B^-$ is symmetric.  The only remaining pairs lie in $B^0\times C^0$, where the two-threshold Ferrers representation gives the required weak sign of $X_i+Y_j$.  These weak inequalities are sufficient because metrizability allows ties among shortest paths.  Thus every prescribed cross-arm path is a geodesic.

All unordered pairs of vertices have now been checked, so the weighting $w$ induces $\cP$.  Hence every neighborly consistent path system on $\Theta_{3,3,3}$ is metrizable, and the non-neighborly systems are metrizable by Proposition~\ref{prop:nonneighborly}.
\end{proof}

\begin{proof}[Proof of Theorem~\ref{thm:main}]
If $a=1$, then $\Theta_{a,b,c}$ is outerplanar, hence metrizable by Proposition~\ref{prop:CLfacts}.  If $a=2$, metrizability follows from Theorem~\ref{thm:theta2bc}.  The remaining positive case $(a,b,c)=(3,3,3)$ is Proposition~\ref{prop:theta333}.

Conversely, if $a\ge3$, $b\ge3$, and $c\ge4$, then non-metrizability follows from Proposition~\ref{prop:nonmetro_side}.  These alternatives are exhaustive when $a\le b\le c$.

For strict metrizability, the cases $a=1$ are covered by outerplanarity and the cases $a=2$ by Theorem~\ref{thm:theta2bc}.  If $a\ge3$, then $\Theta_{a,b,c}$ contains a subdivision of $\Theta_{3,3,3}$, which is not strictly metrizable by Proposition~\ref{prop:CLfacts}(iv).  Topological-minor closure therefore excludes every such theta graph.
\end{proof}

\section{Concluding remarks}

The classification shows that, among theta graphs, the first obstruction is exactly the six-inequality block realized by $\Theta_{3,3,4}$.  All longer theta graphs with three arms of lengths at least $3,3,4$ inherit this obstruction as a topological minor.  On the positive side, a length-two arm suppresses the possible obstruction: the proof for $\Theta_{2,b,c}$ turns consistency into Ferrers relations and realizes them by one-dimensional potentials.

The exceptional case $\Theta_{3,3,3}$ shows that the potential method is not limited to a single short internal vertex.  The two internal vertices of the chosen length-three arm give two nested prefix cuts, and the fact that the remaining central chains have size at most two is exactly what permits the two-threshold Ferrers representation.  The weak inequalities in Lemma~\ref{lem:two_threshold} are essential: they account for the separation between ordinary and strict metrizability.

\begin{corollary}[minimal theta obstructions]\label{cor:minimal-obstructions}
In the topological-minor order restricted to theta graphs, $\Theta_{3,3,4}$ is the unique minimal non-metrizable graph, while $\Theta_{3,3,3}$ is the unique minimal non-strictly-metrizable graph.
\end{corollary}

\begin{proof}
Theorem~\ref{thm:main} shows that every non-metrizable theta graph contains $\Theta_{3,3,4}$ as a subdivision and that every non-strictly-metrizable theta graph contains $\Theta_{3,3,3}$ as a subdivision.  Every proper theta topological minor of the respective graph lies on the positive side of the corresponding classification.
\end{proof}

\begin{corollary}[constructive recognition]\label{cor:algorithm}
The classifications in Theorem~\ref{thm:main} are decidable directly from $(a,b,c)$.  Moreover, given a consistent path system on a positive theta graph, the proofs construct a rational inducing weighting in polynomial time; when $a\le2$, the weighting can be chosen to induce the system strictly.
\end{corollary}

\begin{proof}
The decision test is the pair of explicit conditions in Theorem~\ref{thm:main}.  For realization, the sets $D^-,D^0,D^+$ and the Ferrers row boundaries are obtained by scanning the prescribed paths.  Lemmas~\ref{lem:additive}, \ref{lem:difference}, and \ref{lem:two_threshold} use only ordering, addition, and rational separation of finitely many thresholds.  The cycle reduction contracts persistent edges and evaluates the explicit weights $R_k$.  All steps use polynomially many rational arithmetic operations and yield rational edge lengths.
\end{proof}

\section*{Funding}
This work was supported by the Natural Science Foundation of Shandong Province (Grant No. ZR2024MA073). 

\section*{Data availability}
No new research data were generated or analyzed in this theoretical study. All arguments are contained in the manuscript.

\section*{Declaration of competing interest}
The author declares no known competing financial interests or personal relationships that could have appeared to influence the work reported in this paper.

\section*{Declaration of generative AI and AI-assisted technologies in the manuscript preparation process}
We used ChatGPT to brainstorm the structure, anticipate potential objections, and refine the text. We are fully responsible for all claims, citations, calculations, and the final wording.

\section*{CRediT author statement}
Guangfu Wang: Conceptualization, Methodology, Formal analysis, Investigation, Writing - original draft, Writing - review and editing, Funding acquisition.

\end{document}